\documentclass[a4paper,12pt,reqno]{amsart}
\usepackage{amssymb}
\usepackage{amsmath}
\usepackage{amsthm}
\usepackage{amsfonts}
\usepackage{mathtools}
\usepackage{enumitem}
\usepackage{lipsum}
\makeatletter
\def\dicho#1{\expandafter\@dicho\csname c@#1\endcsname}
\def\@dicho#1{\ifnum#1>1 or \else\fi(\@Roman#1)}
\makeatother
\AddEnumerateCounter{\dicho}{\@dicho}{or (III)}
\newlist{dichotomy}{enumerate}{1}
\setlist[dichotomy]{label=\dicho*,leftmargin=1.5cm}

\usepackage{mathrsfs}
\usepackage[all]{xy}
\usepackage{hyperref}
\usepackage{cleveref}
\usepackage{url}
\usepackage{comment}
\usepackage{xcolor}
\usepackage{pifont}
\usepackage{listings}
\usepackage{todonotes}
\usepackage{tikz-cd}

\usepackage[british]{babel}
\usepackage[margin=1.2in]{geometry}
\usepackage{etoolbox}
\usepackage[title]{appendix}
\usepackage{eqparbox}

\usepackage{pbox}

\numberwithin{equation}{section} \numberwithin{figure}{section}

 \DeclareMathOperator{\End}{End}

\DeclareMathOperator{\red}{red} 

\newcommand{\Q}{\mathbb{Q}}
\newcommand{\F}{\mathbb{F}}
\newcommand{\Z}{\mathbb{Z}}

\newcommand{\OO}{\mathcal{O}}

\newcommand{\rk}{\textup{rk}}

\newcommand{\MF}{\textup{\textsf{MF}}}
\newcommand{\MFC}{\textup{\textsf{MF\textsuperscript{c}}}}
\newcommand{\AMF}{\textup{\textsf{AMF}}}
\newcommand{\AMFC}{\textup{\textsf{AMF\textsuperscript{c}}}}

\theoremstyle{case}

\newtheorem{lemma}{Lemma}

\newtheorem{theorem}[lemma]{Theorem}

\newtheorem{algorithm}[lemma]{Algorithm}
\newtheorem{proposition}[lemma]{Proposition}

\numberwithin{table}{section}

\theoremstyle{definition}
\newtheorem{example}[lemma]{Example}
\newtheorem{question}[lemma]{Question}
\newtheorem{definition}[lemma]{Definition}
\newtheorem{remark}[lemma]{Remark}

\newtheorem*{ack}{Acknowledgements}

\numberwithin{lemma}{section}

\makeatletter
\providecommand\@dotsep{5}
\renewcommand{\listoftodos}[1][\@todonotes@todolistname]{%
  \@starttoc{tdo}{#1}}
\makeatother

\begin{document}

\AtBeginEnvironment{appendices}{\crefalias{section}{appendix}}

\title[Quadratic Cyclic Isogenies]
{Cyclic isogenies of elliptic curves over fixed quadratic fields}

\author{\sc Barinder S. Banwait}
\address{Barinder S. Banwait \\
Dept. of Mathematics \& Statistics\\  
Boston University\\
665 Commonwealth Avenue\\
Boston, MA 02215\\
USA}
\email{barinder.s.banwait@gmail.com}
\urladdr{https://barinderbanwait.github.io}

\author{\sc Filip Najman}
\address{Filip Najman \\
University of Zagreb\\  
Bijeni\v{c}ka Cesta 30 \\
10000 Zagreb\\
Croatia}
\email{fnajman@math.hr}
\urladdr{http://web.math.pmf.unizg.hr/~fnajman}

\author{\sc Oana Padurariu}
\address{Oana Padurariu \\
Dept. of Mathematics \& Statistics\\  
Boston University\\
665 Commonwealth Avenue\\
Boston, MA 02215\\
USA}
\email{oana@bu.edu}
\urladdr{https://sites.google.com/view/oana-padurariu/home}


\begin{abstract}
Building on Mazur's 1978 work on prime degree isogenies, Kenku determined in 1981 all possible cyclic isogenies of elliptic curves over $\Q$. Although more than 40 years have passed, the determination of cyclic isogenies of elliptic curves over a single other number field has hitherto not been realised.

In this paper we develop a procedure to assist in establishing such a determination for a given quadratic field. Executing this procedure on all quadratic fields $\Q(\sqrt{d})$ with $|d| < 10^4$ we obtain, conditional on the Generalised Riemann Hypothesis, the determination of cyclic isogenies of elliptic curves over $19$ quadratic fields, including
$\Q(\sqrt{213})$ and $\Q(\sqrt{-2289})$. To make this procedure work, we determine all of the finitely many quadratic points on the modular curves $X_0(125)$ and $X_0(169)$, which may be of independent interest.
\end{abstract}

\maketitle


\section{Introduction}

An important problem in the theory of elliptic curves over number fields is to understand their possible torsion groups, parametrised by noncuspidal points on the modular curves $X_1(m,n)$, and possible isogenies, parametrised by noncuspidal points on $X_0(N)$. Mazur \cite{mazur1977modular} determined the possible torsion groups over $\Q$ and Kamienny, Kenku and Momose \cite{kamienny92, KM88} determined the possible torsion groups over quadratic fields. More recently, Derickx, Etropolski, van Hoeij, Morrow and Zureick-Brown \cite{Deg3Class} determined the possible torsion groups over cubic fields and Derickx, Kamienny, Stein and Stoll \cite{DKSS} determined all the primes dividing the order of some torsion group over number fields of degree $4 \leq d \leq  7$. Merel proved that the set of all possible torsion groups over all number fields of degree $d$ is finite, for any positive integer $d$ \cite{merel}. All the possible torsion groups over a fixed number field $K$, for many fixed number fields of degree $2$, $3$ and $4$ have been determined, see \cite{najman2010,bruin_najman2016,trbovic2020}, but all the possible isogenies have not been determined over a single number field other than $\Q$.

Unfortunately, much less is known about possible isogeny degrees - for any $d>1$ it is not known what the possible cyclic isogeny degrees of all elliptic curves over all number fields of degree $d$ are. However, the second author \cite{NajmanCamb} determined all the prime degree isogenies of non-CM elliptic curves $E$ with $j(E)\in \Q$ for number fields $d\leq 7$ (and conditionally on Serre's uniformity conjecture for all $d$). This has been extended to all $d>1.4\times 10^7$  unconditionally by Le Fourn and Lemos \cite[Theorem 1.3]{LeFournLemos}.

Mazur \cite{mazur1978rational} determined all the possible primes which arise as degrees of rational isogenies over $\Q$, and explained in the introduction of his paper that to determine all the possible cyclic isogeny degrees, it suffices to determine the $\Q$-rational points on $X_0(N)$ for a small list $S(\Q)$ of composite integers $N$. We will recall this method allowing one to go from prime degree isogenies to composite degree isogenies in \Cref{sec:mazur_strategy}, and will henceforth refer to it as \emph{Mazur's strategy}. Results known at the time allowed Mazur to deal with all but five of these values, viz. $N = 39, 65, 91, 125, 169$. These five cases were subsequently dealt with by Kenku \cite{kenku1, kenku2, kenku3, kenku4}, who showed for these five values that $X_0(N)(\Q)$ consists only of the cusps, yielding the explicit determination of the cyclic isogeny degrees over the rationals. The results are summarized in \Cref{tab:q} below, in which $g$ denotes the genus of $X_0(N)$ and $\nu$ is the number of noncuspidal $\Q$-rational points on $X_0(N)$.

\begin{table}[htp]
\begin{center}
\begin{tabular}{ccccccccccccccccccc}

\hline
$N$ & & $g$ & & $\nu$ & & & $N$ & & $g$ & & $\nu$ & & & $N$ & & $g$ & & $\nu$\\
\hline
$\leq 10$ & & $0$ & & $\infty$ & & & $11$ & & $1$ & & $3$ & & & $37$ & & $2$ & & $2$\\
$12$ & & $0$ & & $\infty$ & & & $14$ & & $1$ & & $2$ & & & $43$ & & $3$ & & $1$\\
$13$ & & $0$ & & $\infty$ & & & $15$ & & $1$ & & $4$ & & & $67$ & & $5$ & & $1$\\
$16$ & & $0$ & & $\infty$ & & & $17$ & & $1$ & & $2$ & & & $163$ & & $13$ & & $1$\\
$18$ & & $0$ & & $\infty$ & & & $19$ & & $1$ & & $1$ & & & & &\\
$25$ & & $0$ & & $\infty$ & & & $21$ & & $1$ & & $4$ & & & & &\\
& & & & & & & $27$ & & $1$ & & $1$ & & & & &\\
\hline
\end{tabular}
\vspace{0.3cm}
\caption{\label{tab:q}Cyclic isogenies over $\Q$.}
\end{center}
\end{table}

\begin{theorem}[Mazur, Kenku]
\Cref{tab:q} is a complete classification of all rational cyclic isogenies of elliptic curves over $\Q$.
\end{theorem}

Since the appearance of Kenku's final paper in 1981, such an explicit determination has not been exhibited for \emph{any other number field}. This is the primary motivation for the present work. 
An important ingredient in our work will be the algorithm to determine isogenies of prime degree for fixed quadratic fields $K$ recently developed by the first author \cite{banwait2021explicit} assuming the Generalised Riemann Hypothesis (GRH). Computing the resulting composite integers $S(K)$ to be treated in Mazur's strategy yields a list typically larger than $S(\Q)$. Although the subject of higher degree points on modular curves has seen much recent development (see e.g. \cite{box2021quadratic,advzaga2021quadratic,NajmanVukorepa,box2021cubic}), some of the values $N \in S(K)$ are such that $X_0(N)$ has large genus, and therefore the determination of the $K$-rational points on $X_0(N)$ is beyond current methods. For this reason, we search for `convenient' quadratic fields $K$ for which (among other conditions) the largest value in $S(K)$ is $169$. This limits the genus of $X_0(N)$ and thereby allows many of the recently developed computational methods to succeed. Our search strategy will be explained in \Cref{sec:search_convenient}.

Among the quadratic fields $\Q(\sqrt{d})$ with $|d| < 10^4$, we find 133 which satisfy our definition of convenient, and therefore, for these quadratic fields, we have some positive hope that we may be able to completely determine the $\Q(\sqrt{d})$-rational points on $X_0(N)$ for $N \in S(\Q(\sqrt{d}))$. Carrying out this program - that is, applying the various known methods (summarised in \Cref{sec:methods_overview}) for determining whether or not $X_0(N)$ has noncuspidal quadratic points over a fixed quadratic field - is the main technical heart of the paper, and yields our main result.

\begin{theorem}\label{thm:main}
Let $d$ be one of the following $19$ values:
\begin{equation}\label{list:really_convenient_ds}
  \begin{gathered}
-6846, \ -2289, \ 213, \ 834, \ 1545, \ 1885, \ 1923,\\
2517, \ 2847, \ 4569, \ 6537, \ 7131, \ 7302,\\
7319, \ 7635, \ 7890, \ 8383, \ 9563, \ 9903.
\end{gathered}
\end{equation}
Then, assuming GRH, \Cref{tab:all_isogs} lists all cyclic isogenies of elliptic curves defined over $\Q(\sqrt d)$ that are not contained in \Cref{tab:q}.
\end{theorem}

\begin{remark}
    \begin{enumerate}
        \item It is interesting to note that only $2$ of the $19$ values we found are negative. See \Cref{sec:search_convenient} for a possible explanation of this.
        \item The assumption of GRH is in fact only necessary for all degree $4$ extensions of $\Q(\sqrt d)$; see \Cref{rem:minimally_finite_assumptions} for more details on this.
    \end{enumerate}
     
\end{remark}

For the convenience of the reader, we provide here the analogous table to \Cref{tab:q} for the smallest absolute value in the list~\eqref{list:really_convenient_ds}.

\begin{table}[htp]
\begin{center}
\begin{tabular}{ccccccccccccccccccc}
\hline
$N$ & & $g$ & & $\nu$ & & & $N$ & & $g$ & & $\nu$ & & & $N$ & & $g$ & & $\nu$\\
\hline
$\leq 10$ & & $0$ & & $\infty$ & & & $14$ & & $1$ & & $2$ & & & $27$ & & $1$ & & $\infty$\\
$12$ & & $0$ & & $\infty$ & & & $15$ & & $1$ & & $\infty$ & & & $32$ & & $1$ & & $\infty$\\
$13$ & & $0$ & & $\infty$ & & & $17$ & & $1$ & & $2$ & & & $36$ & & $1$ & & $\infty$\\
$16$ & & $0$ & & $\infty$ & & & $19$ & & $1$ & & $1$ & & & $37$ & & $2$ & & $2$\\
$18$ & & $0$ & & $\infty$ & & & $20$ & & $1$ & & $\infty$ & & & $43$ & & $3$ & & $1$\\
$25$ & & $0$ & & $\infty$ & & & $21$ & & $1$ & & $4$ & & & $67$ & & $5$ & & $1$\\
$11$ & & $1$ & & $3$ & & & $24$ & & $1$ & & $\infty$ & & & $163$ & & $13$ & & $1$\\
\hline
\end{tabular}
\vspace{0.3cm}
\caption{\label{tab:qsqrt213}Cyclic isogenies over $\Q(\sqrt{213})$, assuming GRH.}
\end{center}
\end{table}

Comparing this with the situation of isogenies over $\Q$, we observe several more values of $N$ for which there are infinitely many elliptic curves with a cyclic isogeny of degree $N$ over our quadratic field. These are all explained by the genus $1$ modular curves which attain positive rank over $\Q(\sqrt{d})$. We do not obtain values of $N$ larger than those over $\Q$ due to our searching for `convenient' quadratic fields to consider.

Obtaining such results for a general number field $K$ (not just quadratic) will require enumeration of the $K$-rational points on some of the same modular curves each time. Indeed, as shown in \Cref{prop:minimally_finite_values} and \Cref{rem:minimally_finite_assumptions}(3), the values $N = 91, 125, 163, 169$ are guaranteed to arise whenever one attempts to enumerate the cyclic isogenies over any given number field. Since, for these values of $N$, the modular curve $X_0(N)$ admits only finitely many quadratic points, it would be valuable to have an explicit answer to the following.

\begin{question}\label{q:get_all_fin_quads}
For $N = 91, 125, 163, 169$, can one determine all of the finitely many quadratic points on $X_0(N)$?
\end{question}

Determining the quadratic points on $X_0(91)$ has recently been done by Vukorepa \cite{Vukorepa}. In this paper we deal with $125$ and $169$, yielding the following result.

\begin{theorem}\label{thm:quad_pts}
All finitely many quadratic points on $X_0(125)$ and $X_0(169)$ are as described in \Cref{tab:quad_125} and \Cref{tab:quad_169} in \Cref{sec:fin_quad_pts}.
\end{theorem}

Of particular note here is the existence of a point on $X_0(125)$ defined over the quadratic field $\Q(\sqrt{509})$ which is not CM and which appears not to have been previously known. Evaluating the $j$-map at this point yields the following $j$-invariant of an elliptic curve over $\Q(\sqrt{509})$ which admits a cyclic $125$-isogeny:\\

\noindent \scalebox{1.1}{$j_{509}=\frac{-2140988208276499951039156514868631437312 \pm 94897633897841092841200334676012564480\sqrt{509}}{161051}.$}\\

This is perhaps surprising because, as $N$ gets larger, the non-CM quadratic points on $X_0(N)$ become rarer; for example we note that, of the curves $X_0(N)$ appearing in \cite{NajmanVukorepa} which have finitely many quadratic points, all of them have either no noncuspidal points or only CM points.  

\subsection*{Contribution to computational methodology}

As in Kenku's work, most of the effort for the determination of cyclic isogenies lies in showing that $X_0(N)(K)$ consists only of the cusps. The methods that we use for this have been implemented in Sage \cite{sagemath}. Moreover, obtaining \Cref{thm:quad_pts} was achieved with the aid of Magma \cite{magma}, and certain parts of the computation were verified in PARI/GP \cite{PARI}. All of the code used in our work may be found at the following GitHub repository:

\begin{center}
\url{https://github.com/barinderbanwait/quadratic_kenku_solver}
\end{center}

Paths to filenames given throughout the paper are relative to the top directory in this repository. In particular, we note that this repository has developed a command-line tool (in \path{sage_code/quadratic_kenku_solver.py}) which automates many of the necessary steps to assist in the determination of cyclic isogenies over a given quadratic field (more details are given at the end of \Cref{sec:isogeny_graphs}).

In summary, we believe that our software contributes to the development of current computational methodology for the following reasons: (1) it uses three computer algebra systems in its implementation; (2) it improves and brings together several other algorithms and methods (outlined in \Cref{sec:methods_overview}) that are utilised through a Python command-line interface; in particular steps previously done manually (e.g. looking up levels $N$ for which there are known quadratic points on $X_0(N)$) have now been fully automated; (3) it is ready-to-use for other researchers to attempt similar determinations of isogenies over other quadratic fields.

\subsection*{Outline of the paper}

\Cref{sec:mazur_strategy} is purely expository, and explains the strategy of Mazur to reduce the problem to determining all $\Q$-rational points on a finite and computable list of modular curves. In \Cref{sec:mazur_strategy_develop} we then generalise this strategy to number fields not containing the Hilbert class field of an imaginary quadratic field. \Cref{sec:search_convenient} describes the searching method to identify the 133 convenient quadratic fields, and \Cref{sec:methods_overview} gives an overview of the methods we used to determine $K$-rational points on modular curves for fixed quadratic fields $K$. While the $19$ values in \ref{list:really_convenient_ds} were found from running our implementation, it would be illuminating for the reader to have one case worked out in detail; this is done in \Cref{sec:example_213} for the quadratic field $\Q(\sqrt{213})$. \Cref{thm:quad_pts} is proved in \Cref{sec:fin_quad_pts}, and \Cref{sec:isogeny_graphs} carries out the computation of the ``graph of rational isogenies'', which is Part 2 of Mazur's strategy. Finally in \Cref{sec:further} we outline avenues for further work into this problem.

\ack{It is a pleasure to thank Jennifer Balakrishnan, Francesca Bianchi, Xavier Guitart, Barry Mazur, Philippe Michaud-Jacobs, Steffen M{\"u}ller, Jeroen Sijsling, Samir Siksek, Michael Stoll and Borna Vukorepa for helpful comments, discussions and correspondence. We thank John Cremona for suggesting improvements to the Sage implementation of some of our algorithms, and Andrew Sutherland for encouraging us to extend the range of quadratic fields to consider. We are grateful to the organisers of the `Modern Breakthroughs in Diophantine Problems' workshop, held in June 2022 at the Banff International Research Station in Banff, Canada, which provided the venue for all authors to meet together physically for the first time and to improve the results of the paper. We also thank the anonymous referees for their careful reading and comments on an earlier version of the manuscript.

The first author is grateful to the Simons Collaboration on Arithmetic Geometry, Number Theory, and Computation for being granted access to computational resources required for this project. The second author was supported by the QuantiXLie Centre of Excellence, a project co-financed by the Croatian Government and European Union
  through the European Regional Development Fund - the Competitiveness
  and Cohesion Operational Programme (Grant KK.01.1.1.01.0004) and by
  the Croatian Science Foundation under the project
  no. IP-2018-01-1313. The third author is supported by NSF grant DMS-1945452 and Simons Foundation grant \#550023. 
}

\section{Mazur's strategy}\label{sec:mazur_strategy}

Mazur's approach to settling the question of rational isogenies for \emph{all} $N$ (not just prime $N$) was based on the following notion: an integer $N$ is said to be \textbf{minimal of positive genus} if the genus of $X_0(N)$ is positive, but the genus of $X_0(d)$ is zero for all proper divisors $d$ of $N$. Note that if $N = 11$ or $N \geq 17$ is prime, then $N$ is minimal of positive genus (and in particular, the set of such integers is infinite).

The strategy then is to carry out the following two computational steps:

\begin{enumerate}[label=(\alph*)]
\item\label{step:min_pos_gen}
one determines $X_0(N)(\Q)$ for all $N$ which are minimal of positive genus, and 
\item\label{step:isogeny_graphs}
for every elliptic curve $E/\Q$ possessing a rational $N$-isogeny, with $N$ minimal of positive genus, one determines the ``graph of rational isogenies'' of $E$.
\end{enumerate}

Here, in Step~\ref{step:isogeny_graphs}, by ``graph of rational isogenies of $E$'' one means to construct the finitely many elliptic curves which are isogenous to $E$ via a rational isogeny, as well as determine the isogenies of minimal degree between the curves in this isogeny class. See \cite[Section 3.8]{cremona1997algorithms} for more on the computations for this over $\Q$.

It is not immediately clear why Step~\ref{step:min_pos_gen} is a finite computation, for two reasons:

\begin{enumerate}
    \item as noted above, there are infinitely many $N$ that are minimal of positive genus;
    \item for a given $N$ that is minimal of positive genus, $X_0(N)(\Q)$ could still be infinite.
\end{enumerate}

For the first point here, we can say the following.

\begin{lemma}\label{lem:desc_min_pos_gen}
    Let $N$ be minimal of positive genus.
    
    \begin{enumerate}
        \item If $N$ is prime, then $N = 11$ or $N \geq 17$.
        \item If $N$ is composite, then $N$ is supported at primes in the set $\left\{2,3,5,7,13\right\}$, and in particular belongs to a finite and explicitly computable set.
    \end{enumerate}
\end{lemma}

\begin{proof}
If $N$ is prime, then the second condition for $N$ being minimal of positive genus is empty, so $N$ will be minimal of positive genus if and only if the genus of $X_0(N)$ is positive, i.e., if $N = 11$ or $N \geq 17$.

If $N$ is composite, then it cannot be divisible by a prime $p = 11$ or $\geq 17$ (because the genus of $X_0(p)$ is positive), so must be supported in the set $\left\{2,3,5,7,13\right\}$. Since the genus of $X_0(k)$ is guaranteed to grow when multiplying $k$ by sufficiently high powers of any prime, one can enumerate the composite $N$ that are minimal of positive genus, and obtain the list shown in \cite[Page 131]{mazur1978rational}.
\end{proof}

This then addresses the first point above: even though there are infinitely many $N$ that are minimal of positive genus, almost all of these are prime. To deal with these one employs Mazur's isogeny theorem in \cite{mazur1978rational} that determines $X_0(p)(\Q)$; in particular, one has that $X_0(p)(\Q)$ consists only of the two cusps if $p \notin \{11, 17, 19, 37, 43,$ $67, 163 \}$; and for $p$ in this set one has a small list of additional noncuspidal rational points (summarised in the table that opens \cite{mazur1978rational}).

To address the second point, one wants to prove the following.

\begin{lemma}\label{lem:finiteness_min_pos_gen}
    Let $N$ be minimal of positive genus. Then $X_0(N)(\Q)$ is finite.
\end{lemma}

\begin{proof}
    We have that the genus of $X_0(N)$ is positive. If it is strictly larger than $1$, then we conclude by Faltings' theorem. If it is of genus $1$, i.e. 
    $$N \in \left\{11, 14, 15, 17, 19, 20, 21, 24, 27, 32, 36, 49\right\},$$ then one directly computes that the elliptic curve $X_0(N)$ has rank $0$.
\end{proof}

\section{Generalising Mazur's strategy to number fields}\label{sec:mazur_strategy_develop}

When attempting to generalise Steps~\ref{step:min_pos_gen} and \ref{step:isogeny_graphs} of Mazur's method to a general number field $K$, one runs into the issue that \Cref{lem:finiteness_min_pos_gen} may no longer be true, since some of the elliptic modular curves $X_0(N)$ may have positive rank over $K$. This suggests that the notion of `minimal of positive genus' be replaced with the following.

\begin{definition}
Let $K$ be a number field. An integer $N$ is said to be \textbf{minimally finite for $K$} if the following conditions are satisfied:
\begin{enumerate}
    \item $X_0(N)(K)$ is finite;
    \item $X_0(d)(K)$ is infinite for all proper divisors $d$ of $N$.
\end{enumerate}
The set of such integers is denoted $\MF(K)$.
\end{definition}

\Cref{lem:finiteness_min_pos_gen} shows that this notion agrees with Mazur's `minimal of positive genus' notion for $K = \Q$, and with this new notion, Steps~\ref{step:min_pos_gen} and \ref{step:isogeny_graphs} become the following:

\begin{enumerate}[label=(\alph*)]
\item\label{step:min_fin}
one determines $X_0(N)(K)$ for all $N$ which are minimally finite for $K$, and
\item\label{step:isogeny_graphs}
for every elliptic curve $E/\Q$ possessing a rational $N$-isogeny, with $N$ minimally finite for $K$, one determines the ``graph of $K$-rational isogenies'' of $E$.
\end{enumerate}

As before, Step (b) is a trivial matter to deal with; indeed, all of Sage \cite{sagemath}, PARI/GP \cite{PARI} and Magma \cite{magma} have fast implementations for computing the isogeny graph of a given elliptic curve; see \Cref{sec:isogeny_graphs} for details on this.

We are thus concerned with carrying out Step~\ref{step:min_fin}, which requires an explicit description of the set $\MF(K)$. The following shows, as before, that there are finitely many composite values in this set.

\begin{lemma}\label{lem:minimally_finite}
    Let $K$ be a number field, and $N$ a positive integer.
    \begin{enumerate}
        \item If $N = 11, 17$ or $19$, then $N \in \MF(K)$ if and only if the rank over $K$ of the elliptic modular curve $X_0(N)$ is zero.
        \item If $N \geq 23$ is prime, then $N \in \MF(K)$.
        \item If $N$ is composite and $N \in \MF(K)$, then $N$ is supported at primes $\leq 19$, and in particular belongs to a finite and explicitly computable set.
    \end{enumerate}
\end{lemma}

\begin{proof}
    Part (1) is clear and Part (2) follows from Faltings's theorem. Suppose now that $N$ is composite and $N \in \MF(K)$. If $N$ were divisible by a prime $p 
    \geq 23$, then by definition of $N \in \MF(K)$ one would have that $X_0(p)(K)$ would be infinite, contradicting Faltings's theorem. The finiteness now follows as in the proof of \Cref{lem:desc_min_pos_gen} (i.e. the genus is guaranteed to grow upon multiplying the level by powers of primes, and then finiteness follows from Faltings's theorem).
\end{proof}

\subsection{Prime minimally finite values}\label{ssec:prime_mf}

To deal with the prime values one requires an analogue of Mazur's isogeny theorem for elliptic curves over $K$, which in this generality does not hold. Indeed, if $K$ contains the Hilbert class field of an imaginary quadratic field $L$, then for any rational prime $p \geq 23$ that splits in $L$ (of which there are infinitely many, of density $\frac{1}{2}$) there exists an elliptic curve $E/L$ with $\End(E) = \OO_L$ that admits a $p$-isogeny rational over the Hilbert class field $H_L$ of $L$, and hence rational over $K$.

Nevertheless it is well-known (see e.g. \cite[Corollary 2]{larson_vaintrob_2014}) that, assuming GRH, this is the only way to obtain infinitely many ``isogeny primes'' for $K$; i.e., if $K$ does not contain the Hilbert class field of an imaginary quadratic field, then the set of primes $p$ for which $X_0(p)$ admits noncuspidal $K$-rational points is finite. Computing a finite set $S(K)$ containing this set of primes is achieved via the program \emph{Isogeny Primes} explained in \cite{banwait2022explicit}, building on work of the first author \cite{banwait2021explicit}. (This program can sometimes, but not always, determine the set exactly.) This allows us to compute a finite list of prime values to be checked in Step~\ref{step:min_fin} as follows.

\begin{algorithm}\label{alg:prime_mf}
Given a number field $K$ not containing the Hilbert class field of an imaginary quadratic field, compute a finite set of primes $T(K)$ as follows.

\begin{enumerate}
    \item (GRH) Compute a superset $S(K)$ of the isogeny primes for $K$ via the program \emph{Isogeny Primes}.
    \item Set \textsf{Output} $:= \left\{p \in S(K) : p \geq 23\right\}$.
    \item (BSD) For $p = 11, 17$ or $19$, if the rank of $X_0(p)$ over $K$ is zero, then append $p$ to \textsf{Output}.
    \item Return \textsf{Output}.
\end{enumerate}
\end{algorithm}

\begin{remark}\label{rem:minimally_finite_assumptions}
    \begin{enumerate}
        \item It is in Step 1 here that we require GRH, to ensure finiteness of the set of isogeny primes for $K$. Specifically, one requires GRH to ensure finitely many isogenies of `Momose Type 2', as in Proposition 6.3 of \cite{banwait2022explicit}. Tracing through the proof of this result, one arrives at a point where one applies the Effective Chebotarev density theorem to a degree $4$ extension $E'$ over $K$. Thus, one only needs to assume GRH for degree $4$ extensions of $K$.
        \item Step 3 requires the rank part of the Birch--Swinnerton-Dyer (BSD) conjecture to honestly be part of an algorithm. Checking whether the analytic rank is $0$ is a finite computation, and if BSD is assumed, this tells us whether the (algebraic) rank is $0$. Moreover, the rank of an elliptic curve can in practice be computed unconditionally in the vast majority of cases. This can be done either by $2$-descent, or if the analytic rank is $0$ or $1$ by applying the results of Gross-Zagier \cite{GrossZagier} and Kolyvagin \cite{Kolyvagin1,Kolyvagin2} which state that the algebraic rank is then $0$ or $1$, respectively.  
        
        \item The primes $37$, $43$, $67$ and $163$ will necessarily always be in the output, because they are isogeny primes over $\Q$.
    \end{enumerate}
\end{remark}

Therefore, to carry out Step~\ref{step:min_fin} above for prime values $N$, we are required to assume that $K$ does not contain the Hilbert class field of an imaginary quadratic field (for otherwise it will not be a finite computation), and moreover one need only determine the finite set $X_0(N)(K)$ for $N$ in the finite set $T(K)$, since for $N \notin T(K)$ we know that either $N$ is not minimally finite (which can only happen if $N \leq 19$), or $X_0(N)(K)$ consists only of the two cusps.

\subsection{Composite minimally finite values}

We now determine the composite minimally finite values for a number field $K$, which we denote $\MFC(K)$; by \Cref{lem:minimally_finite} this set is finite and explicitly computable. We first identify the set of composite levels $N$ that necessarily arise in Step (a) for any number field $K$. We refer to these levels $N$ as \emph{composite always minimally finite}, and denote this set by $\AMFC$.

\begin{lemma}\label{prop:minimally_finite_values}
We have
\[ \AMFC = \left\{26, 35, 39, 50, 65, 91, 125, 169\right\}. \]
\end{lemma}

\begin{proof}
The set of integers $N$ which are minimally finite for \emph{every} number field must have the property that $X_0(N)(K)$ is finite for every number field $K$, but for every proper divisor $d \mid N$, $X_0(d)(K)$ is infinite. By considering the genera of $X_0(N)$ and $X_0(d)$, one sees that the set of such levels $N$ is precisely the subset of the composite integers identified by Mazur as being minimal of positive genus, with the further restriction that $X_0(N)$ does not have genus $1$. Mazur's list is given as 
\[14, 15, 20, 21, 24, 26, 27, 32, 35, 36, 39, 49, 50, 65, 91, 125, 169\]
from which one obtains $\AMFC$ as in the statement above.
\end{proof}

Next, we describe a procedure which allows one to determine $\MFC(K)$ for a given number field $K$.

\begin{algorithm}
Given a number field $K$, compute $\MFC(K)$ as follows.

\begin{enumerate}
    \item \label{step:1} (Genus 1) Compute the set of $N$ for which $X_0(N)$ is an elliptic curve with positive rank over $K$; call this set $B(K)$; denote by $S_1(K)$ the composite values in the complementary set of $N$ for which $X_0(N)$ is an elliptic curve with rank $0$ over $K$.
    \item \label{step:2} (Genus $\geq 2$) Compute the set of products $pb$ for $$p \in \left\{2,3,5,7,11,13,17,19\right\}$$ and $b \in B(K)$; call this set $S_2(K)$;
    \item \label{step:3} Set \textsf{Output} $:= S_1(K) \cup S_2(K) \cup \AMFC$;
    \item \label{step:4} (Remove multiples) Remove multiples from \textsf{Output} (that is, values $y$ in \textsf{Output} for which there exists an $x$ in \textsf{Output} such that $x$ divides $y$);
    \item Return \textsf{Output}.
\end{enumerate}
\end{algorithm}

\begin{proof}

We need to prove that the output of the algorithm is indeed $\MFC(K)$ as claimed.

We first consider minimally finite $N$ for which $X_0(N)$ has genus $1$. This means that $X_0(N)$ is an elliptic curve with rank $0$ over $K$. There are only twelve levels $N$ for which $X_0(N)$ is an elliptic curve, viz.
\begin{equation}\label{eqn:genus_1}
11,14,15,17,19,20,21,24,27,32,36,49,
\end{equation}
and it is readily observed that for each integer $N$ in this list, any proper divisor $d$ of $N$ is such that $X_0(d)$ has genus $0$, so in particular admits infinitely many $K$-rational points. Thus the levels $N$ in this list for which $X_0(N)$ has rank $0$ over $K$ are indeed minimally finite for $K$; this accounts for the set $S_1(K)$ in Step~\ref{step:1}.

We next consider minimally finite composite $N$ for which $X_0(N)$ has genus $> 1$. We observe that the only primes which may divide $N$ are the primes $\leq 19$, since primes $p \geq 23$ are such that the genus of $X_0(p)$ is at least $2$ (and hence admit only finitely many $K$-rational points).

By definition of $N$ being minimally finite, we have that $X_0(d)(K)$ is infinite for all proper divisors $d$ of $N$. There are two ways this could happen: either all of the $X_0(d)$ have genus 0; or at least one of them has genus $1$ and has positive rank over $K$. The first case yields integers which must be contained in $\AMFC$, so we are further reduced to considering the levels $N$ which are multiples of the integers in $B(K)$ computed in Step~\ref{step:1}. Moreover, since any multiple of the integers in \ref{eqn:genus_1} yields an $N$ for which $X_0(N)$ has genus at least $2$, we may restrict to considering only multiples of the integers in $B(K)$ by prime values which, as we observed earlier, are bounded by $19$. This explains the computation happening in Step~\ref{step:2}.

Step~\ref{step:4} is required since the set \textsf{Output} computed in Step~\ref{step:3} may contain multiples within it (that is, one member is a proper multiple of another). These proper multiples are clearly not minimal of positive genus, so are removed before being returned.
\end{proof}

\begin{remark}
\begin{enumerate}
    \item As before, strictly speaking this is only an algorithm upon assumption of the Birch--Swinnerton-Dyer conjecture.
    \item The above algorithm is implemented as \path{minimally_finite(d)} in \path{sage_code/utils.py}.
\end{enumerate}
\end{remark}

\subsection{Unifying the prime and composite minimally finite values}

Having separately treated the prime and composite values, it is helpful to combine the various sets of integers into a single finite list containing the values $N$ that need to be checked in Step~\ref{step:min_fin} of Mazur's strategy. To ensure finiteness of the list of primes to be checked, we restrict the number field as in \Cref{ssec:prime_mf}.

\begin{definition}
    \begin{enumerate}
        \item We define the set of \textbf{always minimally finite} values to be $\AMF := \AMFC \cup \left\{37, 43, 67, 163\right\}$.
        \item For a number field $K$ not containing the Hilbert class field of an imaginary quadratic field, we define the \textbf{finite part of the minimally finite values of $K$} to be \[ \MF^f(K) := \MFC(K) \cup T(K), \]
        where $T(K)$ is the set output by \Cref{alg:prime_mf}.
    \end{enumerate}
\end{definition}

We have defined $\AMF$ to make clear which values will arise in Step~\ref{step:min_fin} for every number field $K$; having a description of higher degree points on the modular curves for the levels in $\AMF$ therefore becomes of some value to the project of uniformly classifying isogenies of elliptic curves over higher degree number fields. 

\begin{example}
Running the program \emph{Isogeny Primes} for the two quadratic fields $\Q(\sqrt{-5})$ and $\Q(\sqrt{5})$ yields that the sets $T(K)$ in each case are $\left\{11, 17, 19, 23\right\}$ and $\left\{11, 19, 23,47\right\}$ respectively. Therefore, running \path{minimally_finite(d)} for $d =-5$ and $5$ we obtain the following values to be checked for each $K$.
\begin{equation}\label{ex:pm5}
  \begin{aligned}
\MF^f(\Q(\sqrt{-5})) &= \AMF \cup \left\{11, 15, 17, 19, 20, 23, 28, 36, 42, \right.\\
&\ \left.48, 54, 63, 64, 81, 98, 147, 343\right\}.\\
\MF^f(\Q(\sqrt{5})) &= \AMF \cup \left\{11, 14, 15, 19, 20, 21, 23, 24, 34,\right.\\
&\ \left.36, 47, 49, 51, 54, 64, 81, 85, 119, 221, 289\right\}.
\end{aligned}
\end{equation}
\end{example}

\section{Searching for convenient quadratic fields}\label{sec:search_convenient}

\Cref{ex:pm5} shows that, to carry out Mazur's strategy for the two quadratic fields $K = \Q(\sqrt{5})$ and $K = \Q(\sqrt{-5})$, it is necessary to determine the $K$-rational points on large genus modular curves such as $X_0(289)$ and $X_0(343)$. Unfortunately, the large genera of these curves is a significant obstacle to employing several of the known methods for determining quadratic points on modular curves.

For this reason we searched through all squarefree integers $-10{,}000 < d< 10{,}000$ and returned $d$ if the following conditions were satisfied:

\begin{enumerate}
    \item $\Q(\sqrt{d})$ is not an imaginary quadratic field of class number one (which ensures that the set of cyclic isogeny degrees is finite);
    \item the values in $\MF^f(\Q(\sqrt{d}))$ larger than 100 are only either the unavoidable $125$, $163$ and $169$, or are values $N$ for which all of the quadratic points on $X_0(N)$ have been determined;
    \item the Mordell-Weil group of the Jacobian of the modular curve $X_0^+(163)$ does not grow when base-extended from $\Q$ to $\Q(\sqrt{d})$.
\end{enumerate}

The third filter here has been employed to enable the `No growth in plus-part' method explained in \Cref{sec:methods_overview} to successfully deal with the case of $163$, which is otherwise a difficult case to surmount. Values $d$ surviving these filters are the `convenient' values mentioned in the Introduction, since the subsequent task of determining all $\Q(\sqrt{d})$-rational points on the finitely many resulting modular curves is made somewhat easier to carry out.

The search algorithm is implemented as \path{search_convenient_d} in \path{sage_code/utils.py}; running it for the range described above yields 133 convenient values of $d$, the smallest of which is $-9946$, the largest is $9995$, and only 26 of which are negative. One possible explanation for the paucity of negative values is that the elliptic modular curve $X_0(49)$ has positive rank over the vast majority of imaginary quadratic fields in our range ($5{,}434$ out of $6{,}083$), while it has rank zero over the majority of real quadratic fields in the range ($4{,}610$ out of $6{,}082$). This means that the value $7^3 = 343$ appears in $\MF^f(\Q(\sqrt{d}))$ far more often for negative values of $d$ than for positive values, and hence (because we do not have a description of all quadratic points on $X_0(343)$) is deemed not convenient by Condition (2) above. We know no good reason why one would expect the rank of $X_0(49)$ over quadratic fields to behave thus.

\section{Overview of the methods used}\label{sec:methods_overview}

This section gives an overview of the methods we employ to determine the $K$-rational points on modular curves (for $K$ a quadratic field). The following notation will be used:
\begin{align*}
  &\begin{aligned}
    \mathllap{K} &= \text{a quadratic field;}\\
  \end{aligned}\\
   &\begin{aligned}
    \mathllap{J_0(N)} &= \text{Jacobian variety of } X_0(N);\\
  \end{aligned}\\
  &\begin{aligned}
    \mathllap{w_N} &= \text{Atkin-Lehner involution on } X_0(N);\\
  \end{aligned} \\
    &\begin{aligned}
    \mathllap{X_0^+(N)} &=X_0(N)/ \langle w_N \rangle;\\
  \end{aligned} \\
    &\begin{aligned}
    \mathllap{J_0(N)_+} &= \text{the sub-abelian variety } (1+w_N)J_0(N);
  \end{aligned}\\
   &\begin{aligned}
    \mathllap{J_0(N)_-} &= \text{the sub-abelian variety } (1-w_N)J_0(N);
  \end{aligned}\\
    &\begin{aligned}
    \mathllap{J_0^+(N)} &= \text{the quotient abelian variety } J_0(N)/J_0(N)_-;
  \end{aligned}\\
     &\begin{aligned}
    \mathllap{J_0^-(N)} &= \text{the quotient abelian variety } J_0(N)/J_0(N)_+.
  \end{aligned}
\end{align*}
Thus $J_0^+(N)$ (respectively $J_0^-(N)$) are quotients of $J_0(N)$ on which $w_N$ acts as $+1$ (respectively $-1$). Moreover, from \cite[Chapter II Section 10]{mazur1977modular}, we have that $J_0^+(N)$ and $J_0(N)_+$ are isomorphic as abelian varieties over $\Q$.

\subsection{Quotient Method}
Given an explicit finite map between curves $\varphi : C \rightarrow D$ defined over $K$ (where a curve is a separated, geometrically integral scheme of dimension 1 of finite type), whenever we are able to determine the $K$-rational points on $D$, we can compute their preimages and so determine $C(K)$. A particularly convenient case is when $D$ is an elliptic curve of rank $0$ over $K$. This is used to deal with $N = 37$ in \Cref{ssec:hyperelliptic}.

\subsection{Catalogues of quadratic points}\label{ssec:quad_pts_catalogue}

We use existing classifications of quadratic points on low-genus modular curves, due to several independent works in recent years; in order of appearance, these are: Bruin-Najman \cite{bruin_najman2016}, Özman-Siksek \cite{ozman2019quadratic}, Box \cite{box2021quadratic}, Najman-Vukorepa \cite{NajmanVukorepa}, and Vukorepa \cite{Vukorepa}. In short, quadratic points on $X_0(N)$ are classified \textbf{non-exceptional} or \textbf{exceptional} according to whether or not they arise as pullbacks of $\Q$-rational points on a quotient $X_0(N)/ \langle w_d \rangle$ ($d$ is usually chosen such that the quotient is of minimal genus). Apart from $N = 37$ which is known to be a special case and is treated separately in Box's paper, there are only finitely many exceptional quadratic points, and the aforementioned works determine all such. In some cases, there are only finitely many quadratic points at all, and these are completely determined in those cases.

\subsection{Özman sieve}
Even when there are infinitely many non-exceptional quadratic points on $X_0(N)$, it is sometimes possible to rule out the existence of such points over fixed quadratic fields. One method for this is based on a result of \"{O}zman \cite{ozman2012points}, and is explained in greater detail as Proposition 7.3 in \cite{banwait2021explicit}. This is used for $N = 65$ in \Cref{ssec:non-hyperelliptic}. The package \emph{Isogeny Primes} mentioned in \Cref{alg:prime_mf} contains an implementation of the Özman sieve.

\subsection{Trbovi\'{c} filter}\label{ssec:najman_trbovic_filter}
Another method to rule out possible non-exceptional quadratic points on $X_0(N)$ is based on work of the second author with Trbovi\'{c} \cite[Theorem 2.13]{najman2021splitting}, and applies in the case that $X_0(N)$ is hyperelliptic of genus $\geq 2$. The aforementioned result lists the primes less than $100$ which must be unramified in any quadratic field $K$ such that $X_0(N)$ admits a $K$-rational point which is not a $\Q$-rational point. These unramified primes have been encoded into \path{sage_code/quadratic_points_catalogue.json}. We refer to this method as the \textbf{Trbovi\'{c} filter}, and use it to deal with all hyperelliptic values we need to consider apart from $N = 37$.

\subsection{The `No growth in plus-part' method}

This is analogous to the `No growth in minus-part' method of \cite[Lemma A.2]{banwait2021explicit}, and may be summarised as follows.

\begin{proposition}\label{prop:no_growth_plus_part}
Let $K$ be a quadratic field, and $N$ an integer such that:
\begin{enumerate}
    \item \label{cond:pos_genus} $X^+_0(N)$ has positive genus;
    \item \label{cond:no_growth_tors} $J_0^+(N)(K)$ has trivial torsion;
    \item \label{cond:no_growth_rank} $\rk(J_0^+(N)(K)) = \rk(J_0^+(N)(\Q))$.
\end{enumerate}
Then:
\begin{enumerate}
    \item $J_0^+(N)(K) = J_0^+(N)(\Q)$;
    \item $X_0^+(N)(K) = X_0^+(N)(\Q)$;
    \item any $K$-rational point on $X_0(N)$ arises as the pull-back of a $\Q$-rational point on $X_0^+(N)$.
\end{enumerate}
\end{proposition}

\begin{proof}
For ease of notation in this proof, we write $J = J_0^+(N)$, which by condition~\ref{cond:pos_genus} is non-trivial. Conditions \ref{cond:no_growth_tors} and \ref{cond:no_growth_rank} tell us that $J(K)$ and $J(\Q)$ are isomorphic as abstract groups. We begin by showing that, in fact, they are equal as sets.

Let $r = \rk J(\Q) = \rk J(K)$, let $V = [v_1,\dots,v_r]$ be a vector of generators of $J(\Q)$, and let $W = [w_1,\dots,w_r]$ be a vector of generators for $J(K)$. Then there exists an $r \times r$ matrix $M$ with coefficients in $\Z$ such that $M W = V$. The matrix $M^{-1}$ may not have coefficients in $\Z$, but there exists an integer $d$ such that $dM^{-1}$ has integer coefficients. Then
$$dW = (dM^{-1}) V,$$
so $dw_i = d \overline{w_i} \in J(\Q)$, where $\overline{w_i}$ denotes the Galois conjugate of $w_i$. Then $d(w_i -\overline{w_i}) = 0$, which implies that either $w_i = \overline{w_i}$ or we have a torsion point. Since the torsion of $J(K)$ is trivial, we conclude that $J(K) = J(\Q)$ as sets, and hence establishing the first conclusion.

Writing $C$ for $X_0^+(N)$, this implies that $C(K) = C(\Q)$. Indeed, fixing a $\Q$-rational point $P$ of $C$ (e.g. the image of any of the $\Q$-rational cusps of $X_0(N)$) and considering the Abel-Jacobi map
\[ \iota_P : C \rightarrow J, \quad Q \mapsto [Q - P], \]
we see that, if $Q \in C(K)$, then $\iota_P(Q) = [Q - P] \in J(K) = J(\Q)$, and hence $Q \in  C(\Q)$, thereby establishing (2). Therefore, we find that any $K$-rational point on $X_0(N)$ in fact arises as the pullback of a $\Q$-rational point on $X_0^+(N)$ under the natural map $X_0(N) \to X_0^+(N)$.
\end{proof}

This is used to deal with $N = 37$ in \Cref{ssec:hyperelliptic}, and $N = 43$ and $163$ in \Cref{ssec:non-hyperelliptic}. The reason we call this the `no growth in plus-part' method is because condition \ref{cond:no_growth_tors} is equivalent to $J_0^+(N)$ gaining no new torsion when base-changed from $\Q$ to $K$; indeed, that $J_0^+(N)$ has only trivial $\Q$-torsion is a theorem of Mazur \cite[Theorem 3]{mazur1977modular}.

\subsection{Symmetric Chabauty with Mordell-Weil sieve}
The determination of all quadratic points on $X_0(N)$ for $N=125$ and $N=169$ is approached using the Box-Siksek method as developed in \cite{box2021quadratic}, and using the improvements developed in \cite{NajmanVukorepa}. The Box-Siksek method is based on a combination of Siksek's relative Symmetric Chabauty method as developed in \cite{siksek} together with a Mordell-Weil sieve. We will give more details of this in \Cref{sec:fin_quad_pts} for our particular setup.

\section{An example - Cyclic isogenies over \texorpdfstring{$\Q(\sqrt{213})$}{Qsqrt213}}\label{sec:example_213}
To illustrate how the methods in the previous section may be used to carry out Step (a) of Mazur's strategy, we provide an extended example with the smallest absolute value of $d$ for which we were able to successfully determine all cyclic isogenies over $\Q(\sqrt{d})$, namely, $d = 213$. The set of minimally finite values to be considered here is as follows.
\begin{equation}\label{ex:213}
  \begin{aligned}
\MF^f(\Q(\sqrt{213})) &= \AMF \cup \left\{11, 14, 17, 19, 21, 30, 40, 45, 48,\right.\\
&\ \left.49, 54, 64, 72, 75, 81\right\}.
\end{aligned}
\end{equation}

Since the determination of rational points on $X_0(N)$ depends in large part on the geometry of the modular curve, we have split the values of $N$ to be considered according to whether the curve is elliptic, hyperelliptic, or non-hyperelliptic. Throughout this section we set $K = \Q(\sqrt{213})$.

\subsection{The elliptic cases}\label{ssec:elliptic}
Here we deal with the values of $N$ in $\MF^f(K)$ for which $X_0(N)$ is an elliptic curve.

\begin{lemma}\label{prop:elliptic}
Let $K = \Q(\sqrt{213})$, and let $N$ be an integer such that $X_0(N)$ has genus 1, viz.
\[11,14,15,17,19,20,21,24,27,32,36,49.\]
Then $J_0(N)(K)_{tors} = J_0(N)(\Q)_{tors}$. In particular, for the genus one values $N$ in $\MF^f(K)$ in \Cref{ex:213}, the $j$-invariants of $K$-rational points on $X_0(N)$ are the same as over $\Q$, which are given in \Cref{tab:elliptic}.

\begin{table}[htp]
\begin{center}
\begin{tabular}{|c|c|}
\hline
$N$ & $j(X_0(N)(\Q))$\\
\hline
$11$ & $-32768, -24729001, -121$\\
$14$ & $-3375, 16581375$\\
$17$ & $\frac{-882216989}{131072}, \frac{-297756989}{2}$\\
$19$ & $-884736$\\
$20$ & \ding{55}\\
$21$ & $\frac{-189613868625}{128}, \frac{3375}{2}, \frac{-140625}{8}, \frac{-1159088625}{2097152}$\\
$36$ & \ding{55}\\
$49$ & \ding{55}\\
\hline
\end{tabular}
\vspace{0.3cm}
\caption{\label{tab:elliptic}The finitely many $j$-invariants of elliptic curves corresponding to $\Q$-rational points on genus one modular curves $X_0(N)$. The symbol \ding{55} means that $X_0(N)(\Q)$ consists only of cusps.}
\end{center}
\end{table}
\end{lemma}

\begin{proof}
The claims here are all readily achieved via Magma computation; for verifying no growth in torsion see the procedure \path{CheckTorsionGrowth}, and for the computation of $j$-invariants see \path{EllipticJInvs}, both in \path{magma_code/utils.m}
\end{proof}

\subsection{The hyperelliptic cases}\label{ssec:hyperelliptic}
In this section we deal with the minimally finite $N$ for which $X_0(N)$ is hyperelliptic of genus at least $2$, which are as follows:

\begin{equation}\label{list:hyperelliptic_213}
26, 30, 35, 37, 39, 40, 48, 50.
\end{equation}

\begin{proposition}\label{prop:hyperelliptic}
For $N$ as in \ref{list:hyperelliptic_213}, we have $X_0(N)(K) = X_0(N)(\Q)$. In particular, for $N \neq 37$, $X_0(N)(K)$ consists only of the cuspidal points, and $X_0(37)(K)$ admits two noncuspidal points defined over $\Q$, with corresponding $j$-invariants:
\[ -9317 \mbox{   and   } -162677523113838677.\]
\end{proposition}

\begin{proof}
The Trbovi\'{c} filter (\Cref{ssec:najman_trbovic_filter}) applies to all of the values we need to consider apart from $N = 37$, and shows, for each such $N$ and $K$, that $X_0(N)(K) = X_0(N)(\Q)$. The verification of this may be found in \path{sage_code/hyperelliptic_verifs.py}. That the $\Q$-rational points on these modular curves consists only of the cusps was already known to Mazur prior to the appearance of \cite{mazur1978rational}.

For $N = 37$ we take the following model for the genus $2$ curve $X_0(37)$:
\[ X_0(37) : y^2 = -x^6 -9x^4-11x^2+37.\]
Taking the quotient of $X_0(37)$ by the isomorphism $(x,y) \mapsto (-x,-y)$ we obtain an elliptic curve $E/\Q$ of rank $0$ over $K$. (The elliptic curve $E$ is $J_0(37)^-$, with Cremona label 37b1.) By looking at the preimages of the finitely many points in $E(K)$ we may conclude that $X_0(37)(K) =  X_0(37)(\Q)$ (see the function \path{ComputePreimages} in \path{magma_code/X037.m}) for this verification). That this latter set admits only two noncuspidal points is a classical result of Mazur and Swinnerton-Dyer \cite[Proposition 2]{mazur1974arithmetic}.
\end{proof}

\subsection{The non-hyperelliptic cases}\label{ssec:non-hyperelliptic}
In this section we deal with the non-hyperelliptic minimally finite $N$, which are as follows:

\begin{equation}\label{list:nonhyperelliptic_213}
43, 45, 54, 64, 65, 67, 72, 75, 81, 91, 125, 163, 169.
\end{equation}

\begin{proposition}\label{prop:non-hyperelliptic}
For $N$ as in \ref{list:nonhyperelliptic_213}, we have $X_0(N)(K) = X_0(N)(\Q)$. In particular, for $N \notin \left\{43, 67, 163\right\}$, $X_0(N)(K)$ consists only of the cuspidal points, and for $N \in \left\{43, 67, 163\right\}$, $X_0(N)$ admits precisely one noncuspidal point, whose corresponding $j$-invariants $j_N$ are given in \Cref{tab:nonhyperelliptic_jinvs}.
\end{proposition}

\begin{table}[htp]
\begin{center}
\begin{tabular}{|c|c|c|c|}
\hline
$N$ & $g$ & $j_N$ \\
\hline
$43$ & $3$ & $-884736000$ \\
$67$ & $5$ & $-147197952000$ \\
$163$ & $13$ & $-262537412640768000$ \\
\hline
\end{tabular}
\vspace{0.3cm}
\caption{\label{tab:nonhyperelliptic_jinvs}The $j$-invariants of the noncuspidal rational points on $X_0(43)$, $X_0(67)$, and $X_0(163)$.}
\end{center}
\end{table}

The proof will occupy the rest of the section. A number of claims are established via Sage or Magma computation; these can respectively be found in \path{sage_code/non_hyperelliptic_verifs.py} or \path{magma_code/non_hyperelliptic_verifs.m}.

\subsubsection{$N = 43$}\label{ssec:43}

We apply \cite[Lemma A.2]{banwait2021explicit}, where we use Steps 5--7 of \cite[Algorithm 7.9]{banwait2022explicit} to show that $J_0(43)_{-}(K) = J_0(43)_{-}(\Q)$. This shows that the only source of $K$-rational points on $X_0(43)$ beyond the $\Q$-rational points must correspond to elliptic curves with CM by an order in $\Q(\sqrt{-43})$. The inbuilt Sage function \path{cm_j_invariants} shows that the $j$-invariant of any such elliptic curve must be defined over $\Q$, whence we find that the only $K$-rational point on $X_0(43)$ corresponds to the $\Q$-rational point, which is a CM-point with discriminant $-43$, whose $j$-invariant is $-884736000$.

The code to verify the computational claims made here may be found in \path{sage_code/non_hyperelliptic_verifs.py}.

\subsubsection{$N = 45, 54, 63, 64, 72, 75, 81$}
This follows directly from \cite[Tables]{ozman2019quadratic}.

\subsubsection{$N = 65$}

\cite[Section 4.5]{box2021quadratic} shows that $X_0(65)$ admits no exceptional quadratic points, and the \"{O}zman sieve shows that it admits no non-exceptional $K$-rational points.

\subsubsection{$N = 67$}
\cite[Section 4.6]{box2021quadratic} determined all quadratic points on $X_0(67)$. From this we obtain that the only $K$-rational point on $X_0(67)$ is the one CM point defined over $\Q$, with $j$-invariant $-147197952000$.

\subsubsection{$N = 91$}
This is dealt by Vukorepa \cite{Vukorepa} where all the quadratic points on $X_0(91)$ are determined, and none are rational over $K$.

\subsubsection{$N = 163$}
We apply \Cref{prop:no_growth_plus_part} to show that any $K$-rational points on $X_0(163)$ must arise as pullbacks of $\Q$-rational points on $X_0^+(163)$. Conditions 2 and 3 are checked in \path{magma_code/NonHyperellipticVerifs.m}. \cite[Section 5.3]{advzaga2021quadratic} shows that $X_0^+(163)(\Q)$ consists only of one cusp, together with CM-points. Therefore, any pullback of these points under the hyperelliptic involution must themselves be either cuspidal or CM points. As in \Cref{ssec:43}, a quick Sage computation reveals that we have only the one $\Q$-rational CM-point known to Mazur.

\subsubsection{$N = 125, 169$}
These are dealt with in \Cref{sec:fin_quad_pts}.

This concludes Step (a) of Mazur's strategy for $\Q(\sqrt{213})$. Step (b) will be carried out in \Cref{sec:isogeny_graphs}.

\section{Quadratic points on $X_0(125)$ and $X_0(169)$}\label{sec:fin_quad_pts}

In this section we compute all the quadratic points on $X_0(125)$ and $X_0(169)$. To do this, we apply the approach developed by the second author with Vukorepa \cite{NajmanVukorepa}, which in turn builds on the results of Box \cite{box2021quadratic} and Siksek \cite{siksek}. 

In particular, we use the same approach which had been used in \cite[Section 7.5.]{NajmanVukorepa} to determine the quadratic points on $X_0(131)$. The idea is to show that all the quadratic points on $X_0(125)$ and $X_0(169)$ map to rational points on $X_0^+(125)$ and $X_0^+(169)$, respectively. After this, all there is to do is to explicitly compute the quadratic points by computing the pullbacks of the rational points (that have been determined previously), with respect to the quotient map $X_0(N)\rightarrow X_0^+(N)=X_0(N)/w_N$.

Let $N=125$ or $169$, let $D_\infty:=0+\infty$ the sum of the two rational cusps of $X_0(N)$, and write $X:= X_0(N)$ and $J(X):=J_0(N)$. Note that $w_N$ acts trivially on $D_\infty$. We use the divisor $D_\infty$ as the base point of the Abel--Jacobi map $\iota\colon X^{(2)} \hookrightarrow J(X)$, which is injective since $X$ is non-hyperelliptic.

Since the ranks of $J(X)(\Q)$ and $(1+w_N)(J(X))(\Q)$ are in these cases equal, we see that $(1-w_N)(J(X)(\Q)) \subseteq J(X)(\Q)_\mathrm{tors}$. Let $p$ be a prime of good reduction for $X$. The following commutative diagram describes the set-up of the sieve:
\begin{equation*}\label{SieveDiagram}  \begin{tikzcd}[sep = large]
X^{(2)}(\Q) \arrow[hook]{r}{\iota}  \arrow{d}{\red_p} &  J(X)(\Q)  \arrow{d}{\red_p} \arrow{r}{1-w_N} & J(X)(\Q)_\mathrm{tors} \arrow[hook]{d}{\red_p}
\\ 
X^{(2)}(\mathbb{F}_p) \arrow{r}{\tilde{\iota}} & J(X)(\mathbb{F}_p) \arrow{r}{1-\tilde{w}_N} &  J(X)(\F_p)
\end{tikzcd} \end{equation*}
Here, $\red_p$ denotes reduction mod $p$ (which we note is injective on $J(X)(\Q)_\mathrm{tors}$), and $\tilde{\iota}$ and $\tilde{w}_N$ are the reductions mod $p$ of $\iota$ and $w_N$ respectively. 

Given a hypothetical unknown quadratic point $Q \in X^{(2)}(\Q)$ that does not map to a rational point on $X_0^+(N)$, we first compute a set $S_Q \subseteq X^{(2)}(\mathbb{F}_p)$ of possibilities for $\red_p(Q)$. In order to construct the set $S_Q$, we consider each point $\red_p(R) \in X^{(2)}(\F_p)$ that is the reduction of a known non-pullback point $R \in X^{(2)}(\Q)$ with respect to (quotienting by) $w_N$. We attempt to prove that $\red_p(Q) \ne \red_p(R)$ by applying a symmetric Chabauty criterion as stated in \cite[Theorem~2.1]{box2021quadratic}. By the commutativity of the diagram, we then have that \[ ( (1-w_N) \circ \iota)(Q) \in W_p := \red_p^{-1}\big(((1-\tilde{w}_N) \circ \tilde{\iota} )(S_Q)\big) \subseteq J(X)(\Q)_\mathrm{tors}.\]
The set $W_p$ is explicitly computable, and although it is obtained by investigating matters modulo $p$, the information it gives is independent of the prime $p$. We aim to find a set of odd primes $\mathcal{P}$ of good reduction for our model such that \begin{equation*}
\bigcap_{p \in \mathcal{P}} W_p = [0] \in J(X)(\Q)_\mathrm{tors}.
\end{equation*}
If this is the case, then it follows that \[ ((1-w_N) \circ \iota)(Q) = (1-w_N)[Q-D_\infty] =  [0]. \]  It follows that $[Q-D_\infty]=w_N([Q-D_\infty])$ and hence, since $w_N$ acts trivially on $D_\infty$, we have that $[Q-w_N(Q)]=0$. Since $X_0(N)$ is non-hyperelliptic, it follows that $Q=w_N(Q)$ (as points in $X^{(2)}(\Q)$). Hence, $Q$ either arises from a pair of quadratic points, each of which is a fixed point of $w_N$, or $Q$ is the pullback of a rational point on $X_0(N)/w_N$ with respect to the quotient map $X_0(N) \rightarrow X_0(N)/w_N$. The sieve is successful using $\mathcal{P}=\{3,7\}$ for $N=125$ and $\mathcal{P}=\{3,5\}$ for $N=169$.


The curve $X_0^+(169)$ is isomorphic to $X_s(13)$, which is known to have 7 rational points by \cite[Theorem 1.1]{balakrishnan2019explicit} and the curve $X_0^+(125)$ is known to have 6 rational points \cite[Section 4]{ArulMuller}.

Hence, after computing the pullbacks of the rational points on $X_0^+(N)$, we have all the quadratic points on $X_0(N)$. Using data provided to us by the authors of \cite{CGPS21} and which can be obtained using \cite[Theorem 3.7]{CGPS21}, we can conclude over which quadratic fields there exist CM points and how many there are. This allows us to conclude that the only non-CM points are the points on $X_0(125)$ defined over $\Q(\sqrt{509})$. Their $j$-invariants are as given in the Introduction.

Our results are below. We list all the quadratic points, where $w$ denotes $\sqrt d$.

\subsection{$X_0(125)$}
Model for $X_0(125)$: 
{\scriptsize{
\noindent
\begin{align*} 
&x_1^2 - 10x_2x_3 + 10x_3x_4 - 9x_4^2 - 28x_4x_5 + 70x_4x_6 -19x_5^2 - 85x_6^2 - x_7^2 - 4x_8^2=0,\\
&x_1x_2 - 5x_2x_3 + 4x_3x_4 - x_4^2 - 14x_4x_5 + 23x_4x_6 -
    6x_5x_6 - 18x_6^2 - x_7x_8 - x_8^2=0,\\
&x_1x_3 - 2x_2x_3 - x_3^2 + 2x_3x_4 + 5x_3x_5 - x_4^2 -
    7x_4x_5 + 13x_4x_6 - 6x_5x_6 - 8x_6^2 - x_8^2=0,\\
&x_1x_4 - x_2x_3 + 3x_3x_5 - x_4x_5 + 5x_4x_6 - 4x_5x_6=0,\\
&x_1x_5 - x_3x_4 + 2x_3x_5 + x_4x_5 + x_4x_6 - x_5x_6 +
    x_6^2=0,\\
&x_1x_6 + x_3x_5 - x_4^2 + x_4x_5 + 2x_4x_6 - 2x_5x_6 +
    2x_6^2=0,\\
&x_1x_8 - x_2x_7 + x_3x_7 - x_4x_7 + x_4x_8 + x_5x_7 -
    x_6x_7 + 3x_6x_8=0,\\
&x_2^2 - 2x_2x_3 - x_3^2 + 2x_3x_4 + 2x_3x_5 - x_4^2 - 8x_4x_5
    + 14x_4x_6 - x_5^2 - 2x_5x_6 - 14x_6^2 - x_8^2=0,\\
&x_2x_4 - x_3^2 + 3x_4x_5 - 4x_4x_6 + 4x_6^2=0,\\
&x_2x_5 - x_4^2 + 2x_4x_5 + x_4x_6 - x_5^2 - x_5x_6 - x_6^2=0,\\
&x_2x_6 - x_3x_5 + x_4x_5 - x_4x_6 + x_5x_6=0,\\
&x_2x_8 - x_3x_7 + x_4x_7 - x_4x_8 + x_6x_7=0,\\
&x_3x_6 - x_4x_5 + x_4x_6 - 2x_6^2=0,\\
&x_3x_8 - x_4x_7 + x_4x_8 - 2x_6x_8=0,\\
&x_5x_8 - x_6x_7=0.
\end{align*}
}}

\noindent Genus of $X_0(125)$: $8$.

\noindent Genus of $X_0^+(125)$: $2$.

\noindent Rational cusps: $P_1:=(1:0 : 0 : 0 : 0 : 0 : 1 : 0), P_2:=(-1:0 : 0 : 0 : 0 : 0 : 1 : 0)$.

\noindent Torsion group of $J_0(125)(\Q):\ \Z/25\Z \cdot [P_1-P_2]$.

\noindent Quadratic points (up to conjugation): See \Cref{tab:quad_125}.
{\scriptsize{
\begin{table}
\begin{center}
\begin{tabular}{ccccc}
Name & $d$ & Coordinates & $j$-invariant & CM \\
\hline
$P_1$ & $-1$ & $(-2w : w/2 : w : w : w/2 : w/2 : 1 : 1)$ & $287496$ & $-16$\\
$P_2$ & $509$ & $(-38w/509 : 21w/509 : -12w/509 : 5w/509 : w/509 : -w/509 : -1 : 1)$ & $j_{509}$ & NO\\
$P_3$ & $-11$ & $(-4w/11 : w/11 : 2w/11 : w/11 : w/11 : w/11 : 1 : 1)$ & $-32768$ &  $-11$\\
$P_4$ & $-1$ & $(w/2 : w/4 : -w/2 : 0 : w/4 : -w/4 : -1 : 1)$ & $1728$ & $-4$\\
$P_5$ & $-19$ & $(0 : -w/19 : 0 : 0 : -w/19 : 0 : 1 : 0)$ & $-884736$ & $-19$
\end{tabular}
\caption{\label{tab:quad_125}The finitely many quadratic points on $X_0(125)$. See the Introduction for the expression for $j_{509}$.}
\end{center}
\end{table}}
}
\medskip

\noindent Primes used in sieve: $3,7$.

\subsection{$X_0(169)$}
Model for $X_0(169)$: 
{\scriptsize{
\noindent\begin{align*} 
&x_1^2 - 12x_3^2 - 4x_3x_4 + 12x_3x_5 - 8x_4^2 - 12x_4x_5 -
    4x_5^2 - x_6^2 + 4x_6x_8 - 4x_7^2 - 4x_8^2=0,\\
&x_1x_2 - 7x_3^2 - 2x_3x_4 + 8x_3x_5 - 3x_4^2 - 8x_4x_5 +
    3x_5^2 - x_6x_7 + 2x_6x_8 - 2x_7^2 - x_8^2=0,\\
&x_1x_3 - 21/8x_3^2 - 9/4x_3x_4 + 13/4x_3x_5 - 25/8x_4^2 -
    11/4x_4x_5 - 45/8x_5^2 + 1/2x_6x_8 - 3/2x_7^2 + 1/2x_7x_8 -
    19/8x_8^2=0,\\
&x_1x_4 - 15/8x_3^2 + 1/4x_3x_4 + 15/4x_3x_5 - 3/8x_4^2 -
    5/4x_4x_5 + 9/8x_5^2 + 1/2x_6x_8 - 1/2x_7^2 - 1/2x_7x_8 -
    1/8x_8^2=0,\\
&x_1x_5 - x_3x_4 - x_4^2 + x_4x_5 - 2x_5^2 - x_8^2=0,\\
&x_1x_7 - x_2x_6 + 2x_4x_6 - 4x_4x_7 + x_4x_8 + 2x_5x_6 +
    3x_5x_7 - 2x_5x_8=0,\\
&x_1x_8 - x_3x_6 + x_4x_6 - 2x_4x_7 + 3x_5x_6 + 2x_5x_7 -
    2x_5x_8=0,\\
&x_2^2 - 3x_3^2 - 2x_3x_4 + 2x_3x_5 - 3x_4^2 - 2x_4x_5 -
    3x_5^2 - x_7^2 - x_8^2=0,\\
&x_2x_3 - 19/8x_3^2 + 1/4x_3x_4 + 11/4x_3x_5 + 1/8x_4^2 -
    13/4x_4x_5 + 21/8x_5^2 + 1/2x_6x_8 - 1/2x_7^2 - 1/2x_7x_8 +
    3/8x_8^2=0,\\
&x_2x_4 - 3/8x_3^2 - 3/4x_3x_4 + 3/4x_3x_5 - 7/8x_4^2 -
    1/4x_4x_5 - 27/8x_5^2 + 1/2x_6x_8 - 1/2x_7^2 + 1/2x_7x_8 -
    13/8x_8^2=0,\\
&x_2x_5 - 1/2x_3^2 + x_3x_5 + 1/2x_4^2 + 3/2x_5^2 + 1/2x_8^2=0,\\
&x_2x_7 - x_3x_6 - 1/2x_4x_8 + 2x_5x_6 + 1/2x_5x_7 -
    2x_5x_8=0,\\
&x_2x_8 - x_4x_6 + x_4x_7 + x_4x_8 - x_5x_7=0,\\
&x_3x_7 - x_4x_6 + 2x_4x_7 - 2x_5x_6 - x_5x_7 + 2x_5x_8=0,\\
&x_3x_8 - x_5x_6 + x_5x_8=0.
\end{align*}}}\noindent Genus of $X_0(169)$: $8$.

\noindent Genus of $X_0^+(169)$: $3$.

\noindent Rational cusps: $P_1:=(1:0 : 0 : 0 : 0 : 1 : 0 : 0), P_2:=(-1:0 : 0 : 0 : 0 : 1 : 0 : 0)$.

\noindent Torsion group of $J_0(169)(\Q):\ \Z/7\Z \cdot [P_1-P_2]$.

\noindent Quadratic points (up to conjugation): See \Cref{tab:quad_169}.
{\scriptsize{
\begin{table}
\begin{center}
\begin{tabular}{ccccc}
\label{tab:quad_169}
Name & $d$ & Coordinates & $j$-invariant & CM \\
\hline
$P_1$ & $-1$ & $(w : 0 : 0 : -w/2 : -w/2 : 1 : 1 : 1)$ & $287496$ & $-16$\\
$P_2$ & $-1$ & $(0 : -w/2 : 0 : w/4 : 3w/4 : 1 : -1 : 1)$ & $1728$ & $-4$\\
$P_3$ & $-3$ & $(-w/3 : w/3 : w/3 : -w/3 : 0 : 1 : 1 : 0)$ & $54000$ & $-12 $\\
$P_4$ & $-3$ & $(-5w/18 : w/6 : w/6 : w/2 : w/9 : 5/2 : 3/2 : 1)$ & $0$ & $-3$ \\
$P_5$ & $-3$ & $(w/3 : 0 : 0 : 0 : -w/3 : 1 : 0 : 1)$ & $-12288000$ & $-27$\\
$P_6$ & $-43$ & $(2w/43 : -w/43 : -2w/43 : 2w/43 : 2w/43 : 0 : 1 : 0)$ & $-884736000$ & $-43$
\end{tabular}
\caption{\label{tab:quad_169}The finitely many quadratic points on $X_0(169)$.}
\end{center}
\end{table}}}
\medskip

\noindent Primes used in sieve: $3,5$.

\section{Computing isogeny graphs}\label{sec:isogeny_graphs}

For a general number field $K$, Step (b) of Mazur's approach is to construct, for the $j$-invariants of each of the noncuspidal $K$-rational points on $X_0(N)$ (for $N \in \MF^f(K)$) found in Step (a), the $K$-rational isogeny graph, and extract any ``unrecorded'' isogenies; that is, isogenies whose degree is a multiple of $N$. (We need only consider $j$-invariants, since whether or not an elliptic curve admits a $K$-rational $N$-isogeny depends only on its $j$-invariant.) As Mazur observed in \cite{mazur1978rational}, this is an easy matter to determine (by ``pure thought''), and is made even easier thanks to the Sage implementation of isogeny graphs due to David Roe (for elliptic curves over $\Q$) and John Cremona (for elliptic curves over number fields). The code that carries this out may be found in \path{sage_code/isogeny_graphs.py}. The results have also been verified in the PARI/GP program, via the optimised \path{ellisomat} implementation based on Billerey's algorithm in \cite{billerey2011criteres}.

\begin{example}
We continue with the example of \Cref{sec:example_213}. To summarise the results of that section, we have that, for the levels $N$ in the following list:
\[ \mathcal{S}_1 = \left\{ N \leq 10\right\} \cup \left\{12, 13, 15, 16, 18, 20, 24, 25, 27, 32, 36\right\}  \]
there are infinitely many elliptic curves over $K = \Q(\sqrt{213})$ supporting a $K$-rational $N$-isogeny, and for $N$ in the following list:
\[ \mathcal{S}_2 = \left\{ 11, 14, 17, 19, 21, 37, 43, 49, 67, 163 \right\} \]
there are only finitely many such elliptic curves; moreover, the $j$-invariants of these curves are given in \Cref{prop:elliptic} and Propositions \ref{prop:hyperelliptic} and \ref{prop:non-hyperelliptic}.  For all other integers $N$, there are no elliptic curves supporting $K$-rational cylcic $N$-isogenies. Running \path{unrecorded_isogenies} in \path{sage_code/isogeny_graphs.py}, we find, just as for Mazur, no ``unrecorded'' isogenies, and may conclude that \Cref{tab:qsqrt213} is complete. We note that, by definition of $N$ being minimally finite for $K$, any proper divisor $d$ of $N$ is such that $X_0(d)$ admits infinitely many $K$-rational points; this explains the values of $N$ in \Cref{tab:qsqrt213} for which $\nu = \infty$. 
\end{example}

Carrying out this computation for the other 18 values listed in \Cref{thm:main}, we obtain tables similar to \Cref{tab:qsqrt213}. In the interest of concision we have chosen to present the tables of isogenies in a more condensed format, which may be found in \Cref{tab:all_isogs}. We only show isogenies (degrees as well as how many up to twist) which are not observed for elliptic curves over $\Q$. For these $19$ values, the additional degrees all correspond to values $N$ for which $X_0(N)$ is an elliptic curve of positive rank, and hence admit finitely many isogenies (up to twist). The computation of all necessary isogeny graphs, as well as the entries of \Cref{tab:all_isogs}, have been automated by the function \path{quadratic_kenku_solver} in \path{sage_code/quadratic_kenku_solver.py}, and the $19$ values may be verified by the function \path{very_convenient_vals} in \path{sage_code/utils.py}.

\begin{table}[htp]
\begin{center}
\begin{tabular}{|c|c|}
\hline
$d$ & $N : \nu_N = \infty$\\
\hline
$-6846$ & $15$, $27$, $32$, $36$\\
$-2289$ & $15$, $21$, $24$, $36$\\
$213$ & $15$, $20$, $24$, $27$, $32$, $36$\\
$834$ & $20$, $21$, $24$, $27$\\
$1545$ & $14$, $15$, $21$, $24$, $27$, $36$\\
$1885$ & $14$, $15$, $20$, $21$, $24$, $32$, $36$\\
$1923$ & $14$, $15$, $20$, $27$\\
$2517$ & $14$, $15$, $20$, $21$, $24$, $27$, $32$, $36$\\
$2847$ & $14$, $15$, $21$, $32$\\
$4569$ & $14$, $27$, $36$\\
$6537$ & $14$, $15$, $20$, $36$\\
$7131$ & $14$, $20$\\
$7302$ & $14$, $15$, $21$, $24$, $32$\\
$7319$ & $15$, $20$, $21$, $24$, $27$, $32$, $36$\\
$7635$ & $14$, $15$, $20$\\
$7890$ & $15$, $20$, $21$, $24$, $27$, $36$\\
$8383$ & $15$, $32$, $36$\\
$9563$ & $21$, $24$, $27$, $36$\\
$9903$ & $14$, $15$, $20$, $32$\\
\hline
\end{tabular}
\vspace{0.3cm}
\caption{\label{tab:all_isogs}Cyclic isogeny degrees for elliptic curves over $\Q(\sqrt{d})$. Here, for concision, we have only shown the isogenies which are not observed for elliptic curves over $\Q$. All of these values arise from genus $1$ modular curves which attain positive rank over $\Q(\sqrt{d})$.}
\end{center}
\end{table}

\section{Further work}\label{sec:further}

We remark in closing that, out of the $133$ values that were identified to be convenient, we were only able to successfully execute Mazur's strategy on $19$ of them. This was due to values of $N$ less than $100$ whose $K = \Q(\sqrt{d})$-rational points we were unable to determine. In particular, the value $N = 43$ posed an obstacle for many values of $d$, particularly when the twisted modular curve $X^d(43)$ was everywhere locally soluble. It is possible that these (and other) values may be treated with an application of the Mordell-Weil sieve, as is done (under additional assumptions which we do not have in our setup) in Section 4.4 of \cite{michaud-jacobs2022fermat}.

Finally, in the spirit of \Cref{q:get_all_fin_quads}, obtaining all finitely many quadratic points on $X_0(163)$ would be of great benefit to obtaining similar determinations over other quadratic fields.






\bibliographystyle{alpha}
\bibliography{references.bib}{}

\end{document}